\documentclass[preprint,12pt]{elsarticle}
\usepackage{float}
\usepackage{amssymb}
\usepackage{mathrsfs,tabularx,indentfirst,graphicx,latexsym,amsthm,amssymb, epsfig, color, bbm,bm}
\usepackage{amsmath}

\allowdisplaybreaks[4]
\newtheorem{theorem}{Theorem}[section]
\newtheorem{lemma}{Lemma}[section]

\newtheorem{remark}{Remark}[section]

\numberwithin{equation}{section}

\newtheorem{ conjecture}{ Conjecture}[section]

\journal{ Journal of Function Spaces }
\usepackage{lineno}
\begin{document}

\begin{frontmatter}



\title{ Cyclic inequalities  involving Cater cyclic function }


\author{JiaJin Wen}\address{ College of Mathematics and Computer Science,
 Chengdu University, Chengdu, Sichuan, 610106, P. R. China.\\ wenjiajin623@163.com}

\author{ TianYong Han}\address{ College of Mathematics and Computer Science,
 Chengdu University, Chengdu, Sichuan, 610106, P. R. China.\\ hantian123\_123@163.com}

 \author{Jun Yuan}\address{(Corresponding author)}\address{ School of Information Engineering, Nanjing Xiaozhuang  University, Nanjing, Jiangsu, 211171,  P. R. China.\\ yuanjun@njxzc.edu.cn }

\begin{abstract}
By means of the mathematical analysis theory,  inequality theory, mathematical induction and the dimension reduction method, under the
proper hypotheses, we establish the following cyclic inequalities:
\[\sum_{i=1}^{n} {a_i^{{a_{n+1-i}}}}\leq\sum_{\text{cyc}:~n}^{1\leq i\leq n} {a_i^{{a_{i + 1}}}}\leq \sum_{i=1}^{n} {a_i^{{a_i}}},~\forall n\geq2.\]
\end{abstract}

\begin{keyword}  cyclic function \sep Cater cyclic function \sep Cater inequality \sep mean.



\MSC[2008]{26D15, 26E60  }
\end{keyword}
\end{frontmatter}


\section{Introduction}
We will use the following  hypotheses and notations throughout the paper:
\[\mathbb{N} \triangleq \left\{ {1, 2, \ldots, m, \ldots } \right\},~\mathbb{Z} \triangleq \left\{ {0, \pm 1, \ldots, \pm m, \ldots } \right\},\]
\[\mathbb{R} \triangleq \left( { - \infty ,\infty } \right),~\left( {{a_1}, {a_2}, \cdots, {a_n}} \right) \in {\mathbb{R}^n},~E_{n} \subseteq {\mathbb{R}^n},\;n \in \mathbb{N},~n\geq 2.\]

If the function  $f:E_{n}\to \mathbb{R}$ satisfy the conditions:
\begin{equation*}
f\left( {{a_{k + 1}},{a_{k + 2}}, \cdots, {a_{k + n}}} \right) = f\left( {{a_1},{a_2}, \cdots, {a_n}} \right),\;\forall k \in \mathbb{Z},
\end{equation*}
 where
\begin{equation}\label{1.}
  {a_k} = {a_i} \Leftrightarrow k \equiv i\;\left( {\bmod \;n} \right),\;\forall k \in \mathbb{Z},\;i = 1, 2, \ldots, n,
\end{equation}
then we say that $f$ is a \textbf{cyclic  function}.

According to the above definition, we know that, under the hypotheses in (\ref{1.}), for any function $\chi :E_{m} \to \mathbb{R}, ~2\leq m\leq n$, the function
\[f:E_{n} \to \mathbb{R},\;f\left( {{a_1},{a_2},\cdots, {a_n}} \right)  \triangleq \sum_{i = 1}^{n } {\chi \left( {{a_{i }},{a_{i + 1}}, \cdots ,{a_{i + m-1}}} \right)}=\sum_{\text{cyc}:~n}^{1\leq i\leq n} {\chi \left( {{a_{i }},{a_{i + 1}}, \cdots, {a_{i + m-1}}} \right)}\]
is a  cyclic  function, where and in the future, $\sum_{\text{cyc}:~n}^{1\leq i\leq n}$ represent the \textbf{cyclic summation}. For example,
\[\sum_{\text{cyc}:~n}^{1\leq i\leq n} {a_i^{{a_{i + 1}}}}={ \sum_{i=1}^{n-1} {a_i^{{a_{i+1}}}}+a_n^{{a_{n+1}}}}={ \sum_{i=1}^{n-1} {a_i^{{a_{i+1}}}}+a_n^{{a_1}}},~\forall n\geq 2.\]

In particular, the function
$f\left( {{a_1},{a_2},\cdots, {a_n}} \right)  \triangleq \sum_{i = 1}^{n }a_{i}$ is a cyclic  function. In general, if $f:E_{n} \to \mathbb{R}$ is a symmetric function, then $f$ is a cyclic  function.

Let $\left( {{a_1},{a_2},\cdots, {a_n}} \right)\in(0,\infty)^{n},~n\geq2.$ Then we say that the function
\begin{equation*}  \text{C}:(0,\infty)^{n}\rightarrow \mathbb{R},~\text{C}\left( {{a_1}, {a_2}, \cdots, {a_{n}}} \right)\triangleq\sum_{\text{cyc}:~n}^{1\leq i\leq n} {a_i^{{a_{i + 1}}}} \end{equation*}
is a \textbf{Cater cyclic function} which is a cyclic function,  and
\[\text{C}_{*}:(0,\infty)^{n}\rightarrow \mathbb{R},~\text{C}_{*}\left( {{a_1}, {a_2}, \cdots, {a_{n}}} \right)\triangleq\sum_{i=1}^{n} {a_i^{{a_{i}}}}\]
with
\[\text{C}^{*}:(0,\infty)^{n}\rightarrow \mathbb{R},~\text{C}^{*}\left( {{a_1}, {a_2}, \cdots, {a_{n}}} \right)\triangleq\sum_{i=1}^{n} {a_i^{{a_{n+1-i}}}}\]
are \textbf{Cater-type cyclic functions} which are also the cyclic functions.

Assume that $f:E_{n} \to \mathbb{R}$ and $g:E_{n} \to \mathbb{R}$ are two cyclic  functions. Then we say that the inequalities
\begin{equation*}  f\left( {{a_1}, {a_2},\cdots, {a_n}} \right) \geq g\left( {{a_1}, {a_2},\cdots, {a_n}} \right)~\text{and}~f\left( {{a_1}, {a_2},\cdots, {a_n}} \right) \leq g\left( {{a_1}, {a_2},\cdots, {a_n}} \right)
 \end{equation*}
are  \textbf{cyclic inequalities}.

 In 1980, F. S. Cater established an interesting cyclic inequality involving power exponents as follows \cite{1}:
\begin{equation}\label{2}
\text{C}\left( {{a_1}, {a_2}, \cdots, {a_{n}}} \right) > 1 + \left( {n - 2} \right)\min \left\{ {a_1^{{a_2}},a_2^{{a_3}}, \cdots, a_{n - 1}^{{a_n}},a_n^{{a_1}}} \right\},~\forall n\geq2,
\end{equation}
which is called as \textbf{Cater inequality}.

Like the periodic function, the cyclic  function is also a class of important functions. Since the  analytic expression  of a cyclic function is very special and complex, so we need  to estimate  the bounds of the cyclic function or establish a cyclic inequality.

Cater inequality is a cyclic inequality,  which is a typical representative of the cyclic inequalities.  Since the analytic expression of the  Cater cyclic function is of extreme particularity and complexity, it is an  interesting topic  that to establish a new Cater-type cyclic inequality.

In this paper, we will study the lower with upper bounds of the Cater cyclic function and establish two Cater-type cyclic inequalities, and display the applications of the dimension reduction method \cite{3,11,9,10} in  the inequality theory.

The research methods of this paper  are based on the mathematical induction and the dimension reduction method. The research tools of this paper  include mathematical analysis theory, inequalities theory and the mean theory \cite{3,11,4,6,10,12,13,7,20,8,16,17,18,15,12.1,21}.

 Our main result are the following Theorems \ref{theorem2.4} and \ref{theorem2.3}.
 \begin{theorem}\label{theorem2.4}  Let $\left( {{a_1}, {a_2}, \cdots ,{a_n}} \right) \in {\left( 0,\infty \right)^n},$ where $n\geq2.$ If
\begin{equation}\label{4.01}
 0<a_1\leq a_2\leq\cdots\leq a_n~\text{with}~a_{1}^{a_{n}}\geq e^{-1},
\end{equation}
 then we have the following  \textbf{Cater-type cyclic inequality}:
\begin{equation}\label{5.01}
\text{C}\left( {{a_1}, {a_2}, \cdots, {a_{n}}} \right)\geq \text{C}^{ *}\left( {{a_1}, {a_2}, \cdots, {a_{n}}} \right).
\end{equation}
Equality in (\ref{5.01}) holds if $n=2$ or ${a_1} = {a_2} =\cdots  = {a_n}.$
\end{theorem}

\begin{theorem}\label{theorem2.3} Let $\left( {{a_1}, {a_2}, \cdots ,{a_n}} \right) \in {\left( 0,\infty \right)^n},$ where $n\geq2.$ If
\begin{equation}\label{4}
 0<a_1\leq a_2\leq\cdots\leq a_n,
\end{equation}
 then we have the following \textbf{Cater-type cyclic inequality}:
\begin{equation}\label{5}
\text{C}\left( {{a_1}, {a_2}, \cdots, {a_{n}}} \right)\leq \text{C}_{*}\left( {{a_1}, {a_2}, \cdots, {a_{n}}} \right).
\end{equation}
Equality in (\ref{5}) holds if and only if ${a_1} = {a_2} =\cdots  = {a_n}.$
\end{theorem}

\section{Proof of  Theorem \ref{theorem2.4}}
\begin{proof} Let $j_{1}j_{2}\cdots j_{n}$ be a permutation \cite{07} of $1,2,\ldots,n$, and let
\begin{equation}\label{0.00}
1\leq i< k\leq n~\text{and}~1\leq j_{i}<j_{k}\leq n .
\end{equation}
By (\ref{4.01}) and (\ref{0.00}), we have
\begin{equation}\label{0.0}
0<a_{1}\leq a_{i}\leq a_{k}\leq a_{n}~\text{and}~0<a_{1}\leq a_{j_{i}}\leq a_{j_{k}}\leq a_{n}.
\end{equation}

We first prove that
\begin{equation}\label{0.1}
a_{i}^{a_{j_{i}}}+a_{k}^{a_{j_{k}}}\geq a_{i}^{a_{j_{k}}}+a_{k}^{a_{j_{i}}}.
\end{equation}

Indeed, if $a_{i}= a_{k}$ or $a_{j_{i}}=a_{j_{k}}$, then (\ref{0.1}) is an equation. Now we assume that
\begin{equation}\label{0.011}0<a_{1}\leq a_{i}< a_{k}\leq a_{n}~\text{and}~0<a_{1}\leq a_{j_{i}}< a_{j_{k}}\leq a_{n}.\end{equation}
We define a auxiliary function as follows:
\[f(t)\triangleq t^{a_{j_{i}}}-t^{a_{j_{k}}},~0<a_{i}\leq t \leq a_{k}.\]
Then inequality (\ref{0.1}) can be rewritten as
\begin{equation}\label{0.2}
f(a_{i})\geq f(a_{k}).
\end{equation}
Sine
\[\frac{\text{d}f(t)}{\text{d}t}< 0,~\forall t\in \left[a_{i}, a_{k}\right]\Rightarrow f(a_{i})\geq f(a_{k}),\]
so we just need to prove that
\begin{equation}\label{0.3}
\frac{\text{d}f(t)}{\text{d}t}< 0,~\forall t\in \left[a_{i}, a_{k}\right].
\end{equation}
Since $t\geq a_{i}\geq a_{1}>0$,
\begin{eqnarray*}
 \frac{\text{d}f(t)}{\text{d}t}&=&a_{j_{i}}t^{a_{j_{i}}-1}-a_{j_{k}}t^{a_{j_{k}}-1}\\
 &=&a_{j_{k}}t^{a_{j_{i}}-1}\left(\frac{a_{j_{i}}}{a_{j_{k}}}-t^{a_{j_{k}}-a_{j_{i}}}\right)\\
 &\leq& a_{j_{k}}t^{a_{j_{i}}-1}\left(\frac{a_{j_{i}}}{a_{j_{k}}}-{a_{i}}^{a_{j_{k}}-a_{j_{i}}}\right)\\
 &\leq& a_{j_{k}}t^{a_{j_{i}}-1}\left(\frac{a_{j_{i}}}{a_{j_{k}}}-{a_{1}}^{a_{j_{k}}-a_{j_{i}}}\right),
\end{eqnarray*}
\[\frac{a_{j_{i}}}{a_{j_{k}}}-{a_{1}}^{a_{j_{k}}-a_{j_{i}}}< 0\Rightarrow (\ref{0.3}),\]
and
\[\frac{a_{j_{i}}}{a_{j_{k}}}-{a_{1}}^{a_{j_{k}}-a_{j_{i}}}< 0\Leftrightarrow \log a_{1}>- \frac{\log a_{j_{k}}-\log a_{j_{i}}}{a_{j_{k}}-a_{j_{i}}},\]
so we just need to prove that
\begin{equation}\label{0.4}
\log a_{1}>- \frac{\log a_{j_{k}}-\log a_{j_{i}}}{a_{j_{k}}-a_{j_{i}}}.
\end{equation}
By (\ref{4.01}), we have
\begin{equation}\label{0.60}
a_{1}^{a_{n}}\geq e^{-1}\Leftrightarrow \log a_{1}\geq -\frac{1}{a_{n}}.
\end{equation}
According to (\ref{0.011}), (\ref{0.60}) and the Lagrange mean value theorem,  there exist $\zeta \in \left(a_{j_{i}},a_{j_{k}}\right)$ such that
\[- \frac{\log a_{j_{k}}-\log a_{j_{i}}}{a_{j_{k}}-a_{j_{i}}}=-\frac{\text{d}\log \zeta}{\text{d}\zeta}=-\frac{1}{\zeta}<-\frac{1}{a_{j_{k}}}\leq -\frac{1}{a_{n}}\leq\log a_{1}\Rightarrow (\ref{0.4})\Rightarrow (\ref{0.3}).\]
Hence inequality (\ref{0.1}) is proved.

Based on the above proof, we know that equality in (\ref{0.1}) holds if and only if $a_{i}= a_{k}$ or $a_{j_{i}}=a_{j_{k}}$.

Next, we prove inequality (\ref{5.01}) as follows.

We define a auxiliary function as follows:
\begin{equation}\label{0.6}
F\left(j_{1}j_{2}\cdots j_{n}\right)\triangleq \sum_{i=1}^{n}a_{i}^{a_{j_{i}}},
\end{equation}
where $j_{1}j_{2}\cdots j_{n}$ is a permutation of $1,2,\ldots,n.$ Now we prove that
\begin{equation}\label{0.7}
F\left(j_{1}j_{2}\cdots j_{n}\right)\geq F\left(n(n-1)\cdots(n-m) j_{m+2}^{*}j_{m+3}^{*}\cdots j_{n}^{*}\right),~\forall m:0\leq m\leq n-1,
\end{equation}
where $j_{m+2}^{*}j_{m+3}^{*}\cdots j_{n}^{*}$ is a permutation of $1,2,\ldots,n-m-1.$

We use the mathematical induction for $m$.

(I) Let $m=0.$ If $j_{1}=n,$ then (\ref{0.7}) is an equality. Assume that $j_{1}<n.$ Then exists $k:2\leq k\leq n$, such that $j_{k}=n$.
Since
\[F\left(j_{1}j_{2}\cdots j_{k}\cdots j_{n}\right)=a_{1}^{a_{j_{1}}}+a_{k}^{a_{j_{k}}}+\sum_{1\leq i\leq n,~i\ne j_{1},j_{k}}a_{i}^{a_{j_{i}}},\]
 by  inequality (\ref{0.1}), we have
 \[a_{1}^{a_{j_{1}}}+a_{k}^{a_{j_{k}}}\geq a_{1}^{a_{j_{k}}}+a_{k}^{a_{j_{1}}}=a_{1}^{a_{n}}+a_{k}^{a_{j_{1}}},\]
 that is,
\[F\left(j_{1}j_{2}\cdots j_{n}\right)=F\left(j_{1}j_{2}\cdots j_{k}\cdots j_{n}\right)\geq F\left(j_{k}j_{2}\cdots j_{1}\cdots j_{n}\right)= F\left(nj_{2}^{*} j_{3}^{*}\cdots j_{n}^{*}\right),\]
where $j_{2}^{*}j_{3}^{*}\cdots j_{n}^{*}=j_{2}\cdots j_{1}\cdots j_{n}$ is a permutation of $1,2,\ldots,n-1.$ Hence inequality (\ref{0.7}) holds when  $m=0.$

(II) Suppose that inequality (\ref{0.7}) holds, where $0\leq m\leq n-2.$ Now we prove that
\begin{equation}\label{0.8}
F\left(j_{1}j_{2}\cdots j_{n}\right)\geq F\left(n(n-1)\cdots(n-m-1) j_{m+3}^{**}j_{m+4}^{**}\cdots j_{n}^{**}\right),
\end{equation}
where $ j_{m+3}^{**}j_{m+4}^{**}\cdots j_{n}^{**}$ is a permutation of $1,2,\ldots,n-m-2.$

We first prove that
\begin{equation}\label{0.9}
F\left(n(n-1)\cdots(n-m) j_{m+2}^{*}j_{m+3}^{*}\cdots j_{n}^{*}\right)\geq F\left(n\cdots(n-m-1) j_{m+3}^{**}j_{m+4}^{**}\cdots j_{n}^{**}\right),
\end{equation}
where $ j_{m+3}^{**}j_{m+4}^{**}\cdots j_{n}^{**}$ is a permutation of $1,2,\ldots,n-m-2.$

Indeed, if $ j_{m+2}^{*}=n-m-1,$ then (\ref{0.9}) is an equality. Assume that $j_{m+2}^{*}<n-m-1.$ Then there exists $k:m+3\leq k\leq n$, such that $j_{k}^{*}=n-m-1$.
By  inequality (\ref{0.1}) and the proof of (I), we have
\begin{eqnarray*}
 &&F\left(n(n-1)\cdots(n-m) j_{m+2}^{*}j_{m+3}^{*}\cdots j_{n}^{*}\right)\\
  &=& F\left(n(n-1)\cdots(n-m) j_{m+2}^{*}j_{m+3}^{*}\cdots j_{k}^{*}\cdots j_{n}^{*}\right) \\
  & \geq & F\left(n(n-1)\cdots(n-m)j_{k}^{*} j_{m+3}^{*}\cdots j_{m+2}^{*} \cdots j_{n}^{*}\right) \\
  & = & F\left(n(n-1)\cdots(n-m)(n-m-1) j_{m+3}^{*}\cdots j_{m+2}^{*} \cdots j_{n}^{*}\right) \\
   & = & F\left(n(n-1)\cdots(n-m)(n-m-1) j_{m+3}^{**}j_{m+4}^{**}\cdots j_{n}^{**}\right),
 \end{eqnarray*}
where $ j_{m+3}^{**}j_{m+4}^{**}\cdots j_{n}^{**}= j_{m+3}^{*}\cdots j_{m+2}^{*} \cdots j_{n}^{*}$ is a permutation of $1,2,\ldots,n-m-2.$ Thus, inequality (\ref{0.9}) is proved.

 By  inequalities (\ref{0.7}) and (\ref{0.9}), we get  the inequality (\ref{0.8}). This ends the proof.

In inequality (\ref{0.7}), set $m=n-1,$ we get
\[F\left(j_{1}j_{2}\cdots j_{n}\right)\geq F\left(n(n-1)\cdots321\right),\]
that is,
\begin{equation}\label{0.5}
\sum_{i=1}^{n}a_{i}^{a_{j_{i}}}\geq \sum_{i=1}^{n}a_{i}^{a_{n+1-i}}=\text{C}^{ *}\left( {{a_1}, {a_2}, \cdots, {a_{n}}} \right).
\end{equation}
Equality in (\ref{0.5}) holds if $j_{1}j_{2}\cdots j_{n}=n(n-1)\cdots321$ or ${a_1} ={a_2} = \cdots  = {a_n}.$

In inequality (\ref{0.5}), set
\[j_{1}j_{2}\cdots j_{n}= 23\cdots (n-1)n1.\]
Then inequality (\ref{0.5}) can be rewritten as inequality (\ref{5.01}). Hence inequality (\ref{5.01}) is proved.

Based on the above proof, we know that equality in (\ref{5.01}) holds if $n=2$ or ${a_1} = {a_2} = \cdots  = {a_n}.$

  The proof of   Theorem \ref{theorem2.4} is completed.
\end{proof}
\begin{remark}\label{remark4.1} Let $n\geq2,$ $1\leq a_1\leq a_2\leq\cdots\leq a_n$ or $e^{-1}\leq a_1\leq a_2\leq\cdots\leq a_n\leq 1$. Then (\ref{4.01}) holds.
According to Theorem \ref{theorem2.4}, inequality (\ref{5.01}) holds.
\end{remark}
\begin{remark}\label{remark4.2} Let $a_{i}=\varepsilon+{(i-1)}/{n},~1\leq i\leq n,~n\geq2,$ where $\varepsilon=0.5173446105249118\cdots $ is the root of the equation
$x^{x+1}=e^{-1},$ where $x\in(0,1).$ Then $0<a_1\leq a_2\leq\cdots\leq a_n$ and
\[a_{1}^{a_{n}}=\varepsilon^{\varepsilon+1-n^{-1}}>\varepsilon^{\varepsilon+1}=e^{-1}.\]
According to Theorem \ref{theorem2.4}, inequality (\ref{5.01}) holds. The  calculation of $\varepsilon$ is based on the \textbf{Mathematica} software. The relevant literatures on proving inequalities  by means of the mathematical software can be see \cite{11,6,7,8}.
\end{remark}
\begin{remark}\label{remark4.3} Based on the  proof of Theorem \ref{theorem2.4} we know that: Under the
 hypotheses of Theorem \ref{theorem2.4}, we have
 \begin{equation}\label{0.10}
\text{C}^{ *}\left( {{a_1}, {a_2}, \cdots, {a_{n}}} \right)=\sum_{i=1}^{n}a_{i}^{a_{n+1-i}}\leq\sum_{i=1}^{n}a_{i}^{a_{j_{i}}}\leq \sum_{i=1}^{n}a_{i}^{a_{i}}=\text{C}_{ *}\left( {{a_1}, {a_2}, \cdots, {a_{n}}} \right),
\end{equation}
where $j_{1}j_{2}\cdots j_{n}$ is a permutation of $1,2,\ldots,n$.
\end{remark}
\begin{remark}\label{remark4.4} For the Cater-type cyclic function $\text{C}^{ *}\left( {{a_1}, {a_2}, \cdots, {a_{n}}} \right)$, we have
\begin{equation}\label{0.11}
\text{C}^{ *}\left( {{a_1}, {a_2}, \cdots, {a_{n}}} \right)>\frac{n}{2},~\forall n\geq2,
\end{equation}
\begin{equation} \label{2.01}
\inf\left\{\text{C}^{ *}\left( {{a_1}, {a_2}, \cdots, {a_{2m}}} \right)\right\}=\text{C}^{ *}\left(0I_{m},I_{m}\right)=m
\end{equation}
and
\begin{equation} \label{2.02}
\inf\left\{\text{C}^{ *}\left( {{a_1}, {a_2}, \cdots, {a_{2m+1}}} \right)\right\}=\text{C}^{ *}\left(0I_{m}, e^{-1}, I_{m}\right)=m+e^{-e^{-1}},
\end{equation}
where $m\in\mathbb{N},$ $I_{m}=(1,1,\cdots, 1)\in\mathbb{R}^{m},$ and $e^{-e^{-1}}=0.6922006275553464\cdots>0.5.$

Indeed, in inequality (\ref{2}), set $n=2,$ we get
\[\text{C}\left( {{a_1}, {a_2}} \right)=a_{1}^{a_{2}}+a_{2}^{a_{1}}>1,~\forall a_{1}>0,~\forall a_{2}>0.\]
 Hence
\begin{equation}\label{2.0}
\text{C}^{ *}\left( {{a_1}, {a_2}, \cdots, {a_{n}}} \right)=\frac{1}{2}\sum_{i=1}^{n}\text{C}\left( {{a_i}, {a_{n+1-i}}} \right)>\frac{1}{2}\sum_{i=1}^{n}1=\frac{n}{2}.
\end{equation}
 Since
\[\text{C}\left( {0, 1} \right)=\text{C}\left( {1, 0} \right)=1,~\inf_{t>0}\left\{\text{C}\left( {{t}, {t}} \right)\right\}=2\inf_{t>0}\left\{t^t\right\}=\text{C}\left( {{e^{-1}}, {e^{-1}}} \right)=2e^{-e^{-1}},\]
by (\ref{2.0}), we get (\ref{2.01}) and (\ref{2.02}).
\end{remark}

\section{Proof of  Theorem \ref{theorem2.3}}

In order to prove Theorem \ref{theorem2.3}, we need to establish several lemmas as follows.

\begin{lemma}\label{lemma3.0} Let $\left( {{a_1}, \cdots ,{a_2}} \right) \in {\left( 0,\infty \right)^n}$ and $n=2.$ Then inequality (\ref{5}) holds. Equality in (\ref{5}) holds if and only if ${a_1} = {a_2}.$
\end{lemma}
 \begin{proof} If $a_{1}= a_2$, then $\mbox{\rm C}\left(a_{1},a_{2}\right)=\mbox{\rm C}_{*}\left(a_{1},a_{2}\right).$ Assume that $a_{1}\ne a_2.$ Without losing of generality, we may assume that $a= a_{1}> a_2=b>0.$

 If $a\geq 1$, then
\begin{eqnarray*}
a^{a-b}> b^{a-b}&\Rightarrow& b^{b}> \frac{a^{b}b^{a}}{a^{a}}\\
&\Rightarrow& \mbox{\rm C}_{*}\left(a_{1},a_{2}\right)-\mbox{\rm C}\left(a_{1},a_{2}\right)=a^{a}+b^{b}-\left(a^{b}+b^{a}\right)\\
&>& a^{a}+\frac{a^{b}b^{a}}{a^{a}}-\left(a^{b}+b^{a}\right)=\frac{\left(a^{a}-a^{b}\right)\left(a^{a}-b^{a}\right)}{a^{a}}\geq0\\
&\Rightarrow& (\ref{5}).
 \end{eqnarray*}

 Now we  assume that $1>a= a_{1}>a_2=b>0.$ By the A-G inequality \cite{17}
 \[\sum_{i=1}^{n}p_{i}x_{i}\geq\prod_{i=1}^{n}x_{i}^{p_{i}},\]
 where $p_{i}>0,~x_{i}>0,~i=1, \ldots,n,~n\geq2,~\sum_{i=1}^{n}p_{i}=1,$ and the logarithm inequalities
 \[\frac{x}{1+x}<\log(1+x)< x,~\forall x: -1<x\ne0,\]
 we have
  \begin{eqnarray*}\frac{\text{d}}{\text{d}t}\left(\frac{\log t}{1-t}\right)&=&(1-t)^2\left[\log(1+t-1)-\frac{t-1}{1+t-1}\right]>0,\forall t:0<t<1\\
  &\Rightarrow& \frac{\log a}{1-a}>\frac{\log b}{1-b}\Leftrightarrow  a^{1-b}-b^{1-a}> 0\\
  &\Rightarrow& b^{a-1}-a^{b-1}> 0\\
  &\Rightarrow&\frac{ab\left(b^{a-1}-a^{b-1}\right)}{a^{b}+b^{a}}=a\times\frac{b^{a}}{a^{b}+b^{a}}-b\times\frac{a^{b}}{a^{b}+b^{a}}>0\\
  &\Rightarrow& \frac{\mbox{\rm C}_{*}\left(a_{1},a_{2}\right)}{\mbox{\rm C}\left(a_{1},a_{2}\right)}=\frac{b^{a}}{a^{b}+b^{a}}\left(\frac{a}{b}\right)^{a}+
  \frac{a^{b}}{a^{b}+b^{a}}\left(\frac{b}{a}\right)^{b}\\
  &\geq&\left(\frac{a}{b}\right)^{a\times\frac{b^{a}}{a^{b}+b^{a}}}\left(\frac{b}{a}\right)^{b\times\frac{a^{b}}{a^{b}+b^{a}}}
  =\left(\frac{a}{b}\right)^{a\times\frac{b^{a}}{a^{b}+b^{a}}-b\times\frac{a^{b}}{a^{b}+b^{a}}}  >1\\
  &\Rightarrow&(\ref{5}).
  \end{eqnarray*}
Hence inequality (\ref{5}) holds when $n=2$, and equality in (\ref{5}) holds if and only if ${a_1} = {a_2}$. This ends the proof of Lemma \ref{lemma3.0}.
\end{proof}

According to the theory of mathematical analysis, we have the following Lemma \ref{lemma3.1}.

\begin{lemma}\label{lemma3.1} Let the function $ f:\left( {\alpha ,\beta } \right) \to \mathbb{R}$ be continuous, and let \[f(\alpha)\triangleq f\left( {\alpha  + 0} \right),~f(\beta)\triangleq f\left( {\beta  - 0} \right).\]
If $f$ has no any minimum points, then we have
\begin{equation}\label{7}
  \mathop {\inf }\limits_{t \in \left( {\alpha ,\beta } \right)} \left\{ {f\left( t \right)} \right\} = \min \left\{ {f\left( {\alpha} \right),f\left( {\beta } \right)} \right\};
\end{equation}
If $t_{1}, \cdots, t_{k}, ~k\geq1,$  are all the minimum points of the function $f$, then we have
\begin{equation}\label{7.01}
  \mathop {\inf }\limits_{t \in \left( {\alpha ,\beta } \right)} \left\{ {f\left( t \right)} \right\} = \min \left\{ {f\left( {\alpha} \right),  f\left( {\beta} \right), f\left(t_{1}\right), \cdots,  f\left(t_{k}\right)} \right\}.
\end{equation}
\end{lemma}

In the following discussion, we define an auxiliary function as follows:
\begin{equation}\label{7.1}
 F:(0,\infty)^{2}\rightarrow\mathbb{R},~ F\left( {x,y} \right) \triangleq \left( {y - x} \right)\log x + \log y - \log x.
\end{equation}

\begin{lemma}\label{lemma3.01} Let $F$ be defined by (\ref{7.1}). Then we have
 \begin{equation}\label{7.2}
F(x,y)>0, ~\forall x,y: 0<x<y<1.
\end{equation}
\end{lemma}
\begin{proof} We arbitrarily fix  $x \in \left( {0,1} \right),$ that is, $x \in \left( {0,1} \right)$ is a constant. Since $y\in(x,1),$ and
$$\;\frac{{\partial F\left( {x,y} \right)}}{{\partial y}} = \log x + \frac{1}{y},\;\;\frac{{{\partial ^2}F\left( {x,y} \right)}}{{\partial {y^2}}} =  - \frac{1}{{{y^2}}} < 0,$$
the function $F:\left( {x,1} \right) \to \mathbb{R}$ is a concave function for the variable $y$, which has no any minimum points. So, by Lemma \ref{lemma3.1}, we have
\begin{equation}\label{12.1}
F\left( {x,y} \right) > \mathop {\inf }\limits_{x < y < 1} \left\{ {F\left( {x,y} \right)} \right\}\; = \min \left\{ {F\left( {x,x} \right),F\left( {x,1} \right)} \right\} = \min \left\{ {0, - x\log x} \right\} = 0.
\end{equation}
This ends the proof of Lemma \ref{lemma3.01}.
\end{proof}

In the following discussion, we define an auxiliary function as follows:

\begin{equation}\label{13}
  \phi: \Omega\rightarrow\mathbb{R},~\phi(x,y,z)\triangleq y^y+z^x-(z^y+y^x),
\end{equation}
where
\begin{equation}\label{14}
 \Omega\triangleq \left\{(x,y,z)\in (0,\infty)^{3}:\max\left\{x,z\right\}\leq y \right\}.
\end{equation}

\begin{lemma}\label{lemma3.4} Let the function $\phi: \Omega\rightarrow\mathbb{R}$ be define by (\ref{13}), and let $(x,y,z)\in\Omega.$ If $y\geq1,$ then we have the inequality
\begin{equation}\label{15}
 \phi(x,y,z)\geq 0.
\end{equation}
Equality in (\ref{15}) holds if and only if $y=z$ or $y=x$.
\end{lemma}
\begin{proof}  The inequality (\ref{15}) can be rewritten as
\begin{equation}\label{16}
  y^{x}\left(y^{y-x}-1\right)\geq z^{x}\left(z^{y-x}-1\right).
  \end{equation}
If $ z\leq1\leq y, $ by $0<\max\left\{x,z\right\}\leq y$, we have
\[ y^{x}\left(y^{y-x}-1\right)\geq 0\geq z^{x}\left(z^{y-x}-1\right)\Rightarrow(\ref{16})\Rightarrow(\ref{15});\]
If $1\leq z\leq y, $ from $0<\max\left\{x,z\right\}\leq y$, we have
\begin{eqnarray*}
 y^{x}\geq z^{x}>0,~y^{y-x}-1\geq z^{y-x}-1\geq 0\Rightarrow  y^{x}\left(y^{y-x}-1\right)\geq z^{x}\left(z^{y-x}-1\right)\Rightarrow(\ref{16})\Rightarrow(\ref{15}).
 \end{eqnarray*}
 Hence (\ref{15}) is proved.

Based on the above proof, we  see that equality in (\ref{15}) holds if and only if $y=z$ or $y=x$.  Lemma \ref{lemma3.4} is proved.
\end{proof}

\begin{lemma}\label{lemma3.5} Let the function $\phi: \Omega\rightarrow\mathbb{R}$ be define by (\ref{13}), and let $(x,y,z)\in\Omega.$ If $0<x\leq z\leq y<1,$ then  inequality  (\ref{15}) also holds. Equality in (\ref{15}) holds if and only if $y=z$.
\end{lemma}

\begin{proof} By the proof of Lemma \ref{lemma3.4}, we  just need to prove (\ref{16}).

Indeed, if $y= z$, then (\ref{16}) is an  equation. Let's assume that $0<x\leq z<y<1.$ Then  $0<x<y<1.$

Now we proved that
\begin{equation}\label{16.1}
  y^{x}\left(y^{y-x}-1\right)> z^{x}\left(z^{y-x}-1\right).
  \end{equation}
Since
\begin{eqnarray*}
(\ref{16.1})&\Leftrightarrow& y^{x}\left(1-y^{y-x}\right)< z^{x}\left(1-z^{y-x}\right)\\
&\Leftrightarrow& x\log y+\log\left(1-y^{y-x}\right)< x\log z+\log\left(1-z^{y-x}\right)\\
&\Leftrightarrow& x(\log y-\log z)< -\log\left(1-y^{y-x}\right)+\log\left(1-z^{y-x}\right)\\
&\Leftrightarrow& x<\frac{-\log\left(1-y^{y-x}\right)+\log\left(1-z^{y-x}\right)}{\log y-\log z},
  \end{eqnarray*}
we see that inequality (\ref{16.1}) can be rewritten as
\begin{equation}\label{17}
  x < \frac{{f\left( y \right) - f\left( z \right)}}{{g\left( y \right) - g\left( z \right)}},
\end{equation}
where
\begin{equation}\label{18}
  f\left( t \right) \triangleq  - \log \left( {1 - {t^{y - x}}} \right),\;g\left( t \right) \triangleq \log t,~0<x\leq z \leq t \leq y < 1,
\end{equation}
and $t$ is independent of $x,y$ and $z$. By (\ref{18}), we have
\begin{equation}\label{18.1}
 \;\frac{{{f'}\left( t \right)}}{{{g'}\left( t \right)}} = \frac{{\left( {y - x} \right){t^{y - x}}}}{{1 - {t^{y - x}}}}=(y-x)\left( \frac{{1}}{{1 - {t^{y - x}}}}-1\right),\;0<x\leq z \leq t \leq y < 1.
\end{equation}
By (\ref{18.1}),  the function $\omega: (0,1)\rightarrow\mathbb{R},~\omega(t)\triangleq{f'}\left( t \right)/{g'}\left( t \right)$, is strictly incremental. So we have
\begin{equation}\label{18.2}
 \frac{{{f'}\left( t \right)}}{{{g'}\left( t \right)}}>\frac{{{f'}\left( z \right)}}{{{g'}\left( z \right)}}\geq\frac{{{f'}\left(x \right)}}{{{g'}\left( x \right)}},
 \;\forall t: 0<x\leq z < t < y < 1.
\end{equation}
According to the Cauchy mean value theorem, (\ref{18}), (\ref{18.1}) and (\ref{18.2}), there exists a $\zeta  \in \left( {z,y} \right) \subset \left( {0,1} \right)$ such that
\begin{equation}\label{19}
 \frac{{f\left( y \right) - f\left( z \right)}}{{g\left( y \right) - g\left( z \right)}} = \frac{{{f'}\left( \zeta  \right)}}{{{g'}\left( \zeta  \right)}} > \frac{{{f'}\left( x \right)}}{{{g'}\left( x  \right)}} = \frac{{\left( {y - x} \right){x^{y - x}}}}{{1 - {x^{y - x}}}}.
\end{equation}
Noting that
\begin{eqnarray*}
\frac{{\left( {y - x} \right){x^{y - x}}}}{{1 - {x^{y - x}}}} > x &\Leftrightarrow&  y{x^{y - x}} - {x^{y - x + 1}} > x - {x^{y - x + 1}}\\
 &\Leftrightarrow& y{x^{y - x}} >x \\
&\Leftrightarrow& \left( {y - x} \right)\log x + \log y - \log x >0,\\
\end{eqnarray*}
that is,
\begin{equation}\label{9.2}
 \frac{{\left( {y - x} \right){x^{y - x}}}}{{1 - {x^{y - x}}}} > x \Leftrightarrow  F\left( {x,y} \right)>0.
\end{equation}
By $0<x<y<1,$ (\ref{19}), (\ref{9.2}) and  Lemma \ref{lemma3.01}, we have
\begin{eqnarray*}
&&F\left( {x,y} \right)>0, ~\forall x,y: 0<x<y<1\\
&\Rightarrow&\frac{{\left( {y - x} \right){x^{y - x}}}}{{1 - {x^{y - x}}}} > x, \\
&\Rightarrow&\frac{{f\left( y \right) - f\left( z \right)}}{{g\left( y \right) - g\left( z \right)}}> \frac{{\left( {y - x} \right){x^{y - x}}}}{{1 - {x^{y - x}}}}> x, \\
&\Rightarrow& \frac{{f\left( y \right) - f\left( z \right)}}{{g\left( y \right) - g\left( z \right)}} > x\\
&\Rightarrow& (\ref{17}) \Rightarrow (\ref{16.1}).
  \end{eqnarray*}
Hence inequality (\ref{16.1}) is proved.

Since $(\ref{16.1}) \Rightarrow (\ref{16})\Rightarrow(\ref{15})$, so the inequality  (\ref{15}) is also proved.

Based on the above proof, we see that equality in (\ref{15}) holds if and only if $y=z$.  Lemma \ref{lemma3.5} is proved.
 \end{proof}

 Now we turn to the proof of Theorem \ref{theorem2.3}.

\begin{proof} We use the mathematical induction for $n$.

When $n = 2$, by Lemma \ref{lemma3.0}, Theorem \ref{theorem2.3} is true.

Suppose that Theorem \ref{theorem2.3} is true for $n$, $n \geq 2.$ Then
\begin{equation}\label{34}
 \text{C}\left( {{a_1}, {a_2}, \cdots, {a_{n}}} \right)-\text{C}_{*}\left( {{a_1}, {a_2},  \cdots, {a_{n}}} \right)\leq 0,
\end{equation}
and equality in (\ref{34}) holds if and only if ${a_1} = {a_2} =  \cdots  = {a_n}$, where
\[0<{{a_1}\leq{a_2}\leq \cdots\leq{a_{n}}}.\] Now we prove that
\begin{equation}\label{30}
 \text{C}\left( {{a_1}, {a_2}, \cdots ,{a_{n+1}}} \right) \leq \text{C}_{*}\left( {{a_1}, {a_2}, \cdots,{a_{n+1}}} \right),
\end{equation}
and equality in (\ref{30}) holds if and only if ${{a_1}={a_2}=\cdots={a_{n+1}}},$ where
\[0<{{a_1}\leq{a_2}\leq \cdots\leq{a_{n+1}}}.\]

Since $0<{{a_1}\leq{a_2}\leq \cdots\leq{a_{n+1}}},$  we have
\begin{equation}\label{31}
\max \left\{ {{a_1},{a_2}, \cdots ,{a_{n + 1}}} \right\} = {a_{n + 1}}.
\end{equation}

For ease of expression, by (\ref{31}), we may assume that
\begin{equation}\label{32}
\left( {x,y,z} \right) \triangleq \left( {{a_1},{a_{n + 1}},{a_n}} \right) \in \Omega,
\end{equation}
where $\Omega$ is defined by (\ref{14}). Since
 \begin{eqnarray*}
  &&\text{C}\left( {{a_1}, {a_2}, \cdots, {a_{n+1}}} \right) - \text{C}_{*}\left( {{a_1}, {a_2}, \cdots,{a_{n+1}}} \right)\\
  &=&\sum\nolimits_{i=1}^{n}{a_i^{{a_{i + 1}}}}+{a_{n+1}^{a_{1}}} -\left(\sum\nolimits_{i=1}^{n}{a_i^{{a_{i }}}}+{a_{n+1}^{a_{n+1}}}\right)\\
 &=&\sum\nolimits_{i=1}^{n-1}{a_i^{{a_{i + 1}}}}+{a_{n}^{a_{n+1}}}+{a_{n+1}^{a_{1}}} -\left(\sum\nolimits_{i=1}^{n}{a_i^{{a_{i }}}}+{a_{n+1}^{a_{n+1}}}\right)\\
 &=&\sum\nolimits_{\text{cyc}:~n}^{1\leq i\leq n}{a_i^{{a_{i + 1}}}}-{a_{n}^{a_{1}}}+{a_{n}^{a_{n+1}}}+{a_{n+1}^{a_{1}}} -\left(\sum\nolimits_{i=1}^{n}{a_i^{{a_{i }}}}+{a_{n+1}^{a_{n+1}}}\right)\\
 &=&\left(\sum\nolimits_{\text{cyc}:~n}^{1\leq i\leq n}{a_i^{{a_{i + 1}}}}-\sum\nolimits_{i=1}^{n}{a_i^{{a_{i }}}}\right)-{a_{n}^{a_{1}}}+{a_{n}^{a_{n+1}}}+{a_{n+1}^{a_{1}}} -{a_{n+1}^{a_{n+1}}}\\
  &=&\left(\sum\nolimits_{\text{cyc}:~n}^{1\leq i\leq n}{a_i^{{a_{i + 1}}}}-\sum\nolimits_{i=1}^{n}{a_i^{{a_{i }}}}\right)
  -\left[y^y+z^x-\left(z^y+y^x\right)\right]\\
  &=&\text{C}\left( {{a_1}, {a_2}, \cdots, {a_{n}}} \right) - \text{C}_{*}\left( {{a_1}, {a_2}, \cdots, {a_{n}}} \right)-\phi(x,y,z),
  \end{eqnarray*}
 we have,
  \begin{equation}\label{33}
\text{C}\left( {{a_1}, \cdots, {a_{n+1}}} \right) - \text{C}_{*}\left( {{a_1}, \cdots, {a_{n+1}}} \right)=\text{C}\left( {{a_1}, \cdots ,{a_{n}}} \right) - \text{C}_{*}\left( {{a_1}, \cdots ,{a_{n}}} \right) -\phi(x,y,z),
\end{equation}
where $\phi(x,y,z)$ is defined by (\ref{13}).

Since $ 0<a_1\leq a_2\leq\cdots\leq a_{n+1},$ we see that (\ref{31}) with (\ref{32}) hold and $x\leq z\leq y$. According to the inductive hypothesis, we see that (\ref{34}) holds. By Lemmas \ref{lemma3.4} and \ref{lemma3.5}, we know that (\ref{15}) holds. According to the (\ref{15}), (\ref{34}) and (\ref{33}), we have
\begin{eqnarray*}
&& \text{C}\left( {{a_1}, {a_2},  \cdots, {a_{n+1}}} \right) - \text{C}_{*}\left( {{a_1}, {a_2},  \cdots, {a_{n+1}}} \right)\\
&=&\text{C}\left( {{a_1}, {a_2},  \cdots ,{a_{n}}} \right) - \text{C}_{*}\left( {{a_1}, {a_2},  \cdots ,{a_{n}}} \right) -\phi(x,y,z)\\
&\leq& -\phi(x,y,z)\leq0\Rightarrow (\ref{30}).
  \end{eqnarray*}
Hence inequality (\ref{30}) is proved.

According to the inductive hypothesis and Lemmas \ref{lemma3.4} with \ref{lemma3.5}, we  see that equality in (\ref{30}) holds if and only if
${a_1} = {a_2} = \cdots  = {a_n}~\text{and}~a_{n + 1}=y=z={a_n}$, i.e. ${a_1} = {a_2} = \cdots  = a_{n+1}.$

This completes the proof of   Theorem \ref{theorem2.3}.
\end{proof}
\begin{remark} \label{remark3.1} The proof method of  Theorem \ref{theorem2.3} is called the dimensionality reduction method. The relevant literatures on proving inequalities by means of the  dimensionality reduction method can be see \cite{3,11,9,10}. The dimension reduction process of the proof is as follows. \\
\textbf{(A)} Inequality (\ref{30}) contain $n+1$ variables and  inequality (\ref{5}) contain $n$ variables. We  transform inequality  (\ref{30})  into inequality  (\ref{5}) by mean of the mathematical induction. This transformation process is based on the inequality (\ref{15}). \\
\textbf{(B)} Inequality (\ref{15}) contain three variables $x,y,z.$   By mean of Lemmas \ref{lemma3.1}-\ref{lemma3.5}, we transform  inequality  (\ref{15})  into inequality (\ref{7.2}), which contain only two variables.\\
\textbf{(C)} Lemmas \ref{lemma3.1} and \ref{lemma3.01} transform inequality  (\ref{7.2})  into inequality  (\ref{12.1}). There is only one variable at the right end of the inequality (\ref{12.1}).\\
\textbf{(D)} For an inequality with only one variable, we use mathematical analysis theory to deal with it.\\
\end{remark}
\begin{remark}\label{remark3.2} Let the function $f:[0,1]\rightarrow(0,\infty)$ be continuous and incremental, and let $\lim_{n\rightarrow\infty}n^{-1}\text{C}\left( {{a_1}, {a_2}, \cdots, {a_{n}}} \right)$ exists, where $a_{i}\triangleq f\left({i}/{n}\right),$ $i=1,2,\ldots,n,$ $n\geq 2$, and $n^{-1}\text{C}\left( {{a_1}, {a_2}, \cdots, {a_{n}}} \right)$  is the mean of the positive real numbers $a_1^{{a_2}}, a_2^{{a_3}}, \cdots, a_{n - 1}^{{a_n}}, a_n^{{a_1}}.$ Then, by Theorem \ref{theorem2.3} and the mathematical analysis theory, we have the following inequality:
\begin{equation}\label{02}
\lim_{n\rightarrow\infty}n^{-1}\text{C}\left( {{a_1}, {a_2}, \cdots, {a_{n}}} \right)\leq\int_{0}^{1}\left[f\left(t\right)\right]^{f\left(t\right)}\text{d}t,
\end{equation}
where $\int_{0}^{1}\left[f\left(t\right)\right]^{f\left(t\right)}\text{d}t$ is the mean of the function $f^{f}.$
The relevant literatures on  mean  theory can be see \cite{3,11,4,6,10,12,13,7,20,8,16,17,18,15,12.1,21}.
\end{remark}

\textbf{Competing interests.} The authors declare that they have no conflicts of interest in
this joint work.

\textbf{Authors contributions.} All authors contributed equally and significantly in this
paper. All authors read and approved the final manuscript.

\textbf{Acknowledgements.} The authors would like to acknowledge the support from the
National Natural Science Foundation of China (No. 11161024).

\end{document}